\documentclass[11pt]{article}

\usepackage[utf8]{inputenc}
\usepackage{amsmath,amssymb,amsthm,amsfonts}
\usepackage{mathrsfs}
\usepackage[hidelinks]{hyperref}
\usepackage[margin=1in]{geometry}

\usepackage[backend=bibtex,style=alphabetic]{biblatex}
\bibliography{references}

\newtheorem{theorem}{Theorem}[section] 
\newtheorem{definition}[theorem]{Definition}

\newtheorem{proposition}[theorem]{Proposition}
\newtheorem{corollary}[theorem]{Corollary}

\theoremstyle{remark}

\title{\textbf{The flag major index on conjugacy classes\\ of wreath products}}
\author{Harel Rozenfeld
\thanks{Department of Mathematics, Bar-Ilan University, Ramat-Gan 52900, Israel. {\tt harelrozen@gmail.com}}}
\date{26/10/26}

\begin{document}

\maketitle

\begin{abstract}
A classical result, independently proved by Garsia and by Wachs, gives a simple and elegant closed expression for the  
evaluation, at a primitive root of unity, of the major index generating function over a conjugacy class in the symmetric group. In this paper, we  generalize this result to the flag major index generating function over conjugacy classes in any wreath product \(\mathbb{Z}_{r}\wr S_{n}\). 
Our approach is new, and applies
  conjugation by a
Coxeter element as  a cyclic group action on the conjugacy class. This  leads to a simple formula 
in the general case.
\end{abstract}

\tableofcontents

\section{Introduction}

The study of permutation statistics over conjugacy classes is a classical topic in algebraic combinatorics. One of the most important statistics is the \emph{major index}, defined on the symmetric group $S_{n}$ as the sum of 
the descent positions of a permutation. A central problem is to evaluate, for a given conjugacy class $C \subseteq S_{n}$, the sum of $\omega_{n}^{\mathrm{maj}(\pi)}$ over all $\pi \in C$, where $\omega_{n}$ is a primitive $n$th root of unity. 

This problem has been studied in two complementary ways. 
Garsia~\cite{gar90} provided an algebraic approach using representation theory, while Wachs~\cite{wachs-1991} gave a purely combinatorial treatment. Both showed that when one specializes the generating function of the major index at roots of unity, elegant and surprisingly simple evaluations emerge, depending only on the structure of the conjugacy class. These results form the foundation of the subject.

For reference, we recall their result in a form relevant to our work. Throughout this paper, for every positive integer $m$,
we write $\omega_m=e^{2\pi i/m}$. 
This choice of primitive root is relevant in the wreath-product setting: when $r>2$, replacing $\omega_{rn}$ by a different primitive $rn$-th root of unity may change the value of the sum on conjugacy classes for which it is nonzero. Thus, all root-of-unity evaluations below are understood with the above normalization. Throughout the paper, we assume \(n\geq 2\).

\medskip

\noindent\textbf{Theorem (Garsia, Wachs).}
\emph{Let $C$ be a conjugacy class of $S_{n}$, and let $\omega_{n}$ be a primitive 
$n$th root of unity. Then
\[
\sum_{\pi \in C} \omega_{n}^{\mathrm{maj}(\pi)} =
\begin{cases}
\mu(k), & \text{if } C = k^{\,n/k}, \\
0, & \text{otherwise},
\end{cases}
\]
where $\mu$ is the Möbius function and $C = k^{\,n/k}$ denotes the class of permutations consisting of $n/k$ disjoint cycles, each of length $k$.}


\medskip

In this paper we address the natural generalization of this result to the wreath product $\mathbb{Z}_{r}\wr S_{n}$, the group of colored permutations. For this group the relevant statistic is the \emph{flag major index}, introduced by Adin and 
 Roichman~\cite{adin-roichman}, which incorporates both descent positions and colors. 
The flag major index has been studied extensively in algebraic and enumerative combinatorics, particularly in connection with 
invariant theory and permutation statistics. For 
specializations at roots of unity see, e.g., \cite{AGR05} and follow-ups; for  
systematic developments see, e.g., \cite{FH05,HLR05, BC04};  
   for a recent treatment see \cite{LY25} and references therein. 

\medskip

Our main result is the following theorem.

\medskip
\noindent\textbf{Theorem.}
    Let \(C\) be a conjugacy class of \(\mathbb{Z}_{r}\wr S_{n}\). Then
    \[
        \sum_{g\in C}\omega_{rn}^{\,\mathrm{fmaj}(g)} \;=\;
        \begin{cases}
        \omega_{r}^{c/k} \,\mu(k), & \text{if } C = k_{c}^{\,n/k} \text{ and } \gcd(r,k)=1,\\[4pt]
        0, & \text{otherwise},
        \end{cases}
    \]
   where \(k_c^{\,n/k}\) denotes the conjugacy class in which every cycle has length \(k\) and color \(c\). In the first case, \(c/k\) denotes \(c\,k^{-1}\in\mathbb{Z}_r\), where \(k^{-1}\) is the inverse of \(k\) modulo \(r\).
\medskip

The approach taken in this work is based on analyzing the action of
conjugation by Coxeter elements in \(S_n\) and $\mathbb{Z}_{r}\wr S_{n}$. This point of view allows us to recover the symmetric group case in a uniform way and then extend the argument to $\mathbb{Z}_{r}\wr S_{n}$.

The structure of the paper is as follows. In Section~2 we provide background on permutation statistics, colored permutations, and the flag major index, including cyclic descents. In Section~3 we develop a new proof for the 
symmetric group, based on conjugation by the Coxeter element. In Section~4 we generalize this method to wreath products, thereby establishing the theorem stated above.

\section{Preliminaries}

We begin by recalling the basic objects and notations that will be used throughout the paper. Our main focus will be on the group of \emph{colored permutations}, also known as the wreath product \(\mathbb{Z}_{r}\wr S_{n}\). We also introduce the relevant permutation statistics, in particular the major index and the flag-major index.

\subsection{The symmetric group}

Let \(S_{n}\) denote the symmetric group on \([n]=\{1,2,\dots,n\}\). An element \(\sigma \in S_{n}\) is written in one-line notation as

\[
\sigma = [\sigma(1),\sigma(2),\dots,\sigma(n)].
\]

For \(\sigma \in S_{n}\), the \emph{descent set} is defined by

\[
\mathrm{Des}(\sigma) := \{\, i \in [n-1] : \sigma(i) > \sigma(i+1) \,\}.
\]

The \emph{major index} of \(\sigma\) is then

\[
\mathrm{maj}(\sigma) := \sum_{i \in \mathrm{Des}(\sigma)} i.
\]

\subsection{Colored permutations}

Fix integers \(n \geq 2\) and \(r \geq 1\).  
The group of \emph{colored permutations} is the wreath product

\[
\mathbb{Z}_{r}\wr S_{n} = \{\, (c,\sigma) : c=(c_{1},\dots,c_{n}) \in \mathbb{Z}_{r}^{n},\; \sigma \in S_{n}\,\}.
\]
We write an element \(g=(c,\sigma)\) in one-line form as 

\[
g = [g(1),\dots,g(n)] = [\,\sigma(1)^{c_{1}},\;\sigma(2)^{c_{2}},\;\dots,\;\sigma(n)^{c_{n}}\,],
\]
where the superscript denotes the color \(c_{i}\in \mathbb{Z}_{r}\) attached to the entry \(\sigma(i)\).

\medskip
If \(C=(i_1\,i_2\,\ldots\,i_\ell)\) is a cycle of \(\sigma\), we define its
\emph{color} by
\[
\operatorname{col}(C):=
c_{i_1}+c_{i_2}+\cdots+c_{i_\ell}\in\mathbb Z_r.
\]
Thus the sum is taken over the positions belonging to the cycle (equivalently,
over the values appearing in it).

Two elements of \(\mathbb Z_r\wr S_n\) are conjugate if and only if, for every
\(\ell \geq 1\) and \(c\in\mathbb Z_r\), they have the same number of cycles of length
\(\ell\) and color \(c\).

\paragraph{Matrix representation.}
Each element \(g=(c,\sigma)\) can be represented as an \(n\times n\) \emph{monomial matrix} over the complex numbers.
Let \(\omega_{r}=e^{2\pi i/r}\).  
Then define the matrix \(\Phi(g)\) by
\[
\bigl(\Phi(g)\bigr)_{i,j}=
\begin{cases}
\omega_{r}^{\,c_{j}}, & i=\sigma(j),\\
0, & \text{otherwise}.
\end{cases}
\]
Thus each column contains exactly one nonzero entry, equal to an \(r\)-th root of unity.

\paragraph{Example.}
For instance, let \(r=3,n=3\), take \(\sigma=(1\,2\,3)\) and \(c=(1,0,2)\).  
Then
\[
g = [2^{1},3^{0},1^{2}], \qquad
\Phi(g)=
\begin{pmatrix}
0 & 0 & \omega_{3}^{2} \\
\omega_{3}^{1} & 0 & 0 \\
0 & 1 & 0
\end{pmatrix}.
\]

\paragraph{Multiplication.}
We define multiplication in \(\mathbb{Z}_{r}\wr S_{n}\) so that
\[
\Phi(g h) \;=\; \Phi(g)\Phi(h).
\]
In other words, the group law of \(\mathbb{Z}_{r}\wr S_{n}\) is exactly the one induced from
matrix multiplication in this faithful monomial representation.
Equivalently, if \(g=(c,\sigma)\) and \(h=(c',\sigma')\), then
\[
gh=(c'',\sigma\sigma'),
\]
where
\[
c''_j=c'_j+c_{\sigma'(j)}.
\]

\subsection{Order on colored numbers}

To define descent-type statistics, we impose a linear order on the set of colored numbers.  
For each \(a \in [n]\) and color \(d \in \mathbb{Z}_{r}\), we write the symbol \(a^{d}\).  
The order is given by
\[
1^{r-1} \prec 2^{r-1} \prec \cdots \prec n^{r-1} \prec \cdots \prec 1^{1} \prec \cdots \prec n^{1} 
\prec 1^{0} \prec 2^{0} \prec \cdots \prec n^{0}.
\]
Intuitively, higher colors are considered smaller, and within the same color we compare the underlying integers.

\subsection{Descents and major index in
\texorpdfstring{\(\mathbb{Z}_{r}\wr S_{n}\)}{Zr wreath Sn}}

Given \(g=[\sigma(1)^{c_{1}},\dots,\sigma(n)^{c_{n}}]\in \mathbb{Z}_{r}\wr S_{n}\), we define the descent set
\[
\mathrm{Des}(g) := \{\, i \in [n-1] : g(i) \succ g(i+1) \,\},
\]
where the comparison is with respect to the order defined above.  
The \emph{major index} of \(g\) is then
\[
\mathrm{maj}(g) := \sum_{i \in \mathrm{Des}(g)} i.
\]

\subsection{The flag-major index}

The \emph{flag-major index} was introduced by Adin and Roichman~\cite{adin-roichman} 
as a natural generalization of the classical major index to the group 
of colored permutations \(\mathbb{Z}_{r}\wr S_{n}\). 
For \(g=(c,\sigma) \in \mathbb{Z}_{r}\wr S_{n}\) with 
\(c=(c_{1},\dots,c_{n}) \in \mathbb{Z}_{r}^{n}\), we identify each color
\(c_i\in\mathbb{Z}_r\) with its standard representative in
\(\{0,1,\dots,r-1\}\). With this convention, the flag-major index is the
nonnegative integer
\[
\mathrm{fmaj}(g) := r \cdot \mathrm{maj}(g) + \sum_{i=1}^{n} c_{i}.
\]

\subsection{Cyclic descents}

In addition to ordinary descents, it is often useful to consider
\emph{cyclic descents}.
For a colored permutation
\[
g=[\sigma(1)^{c_1},\dots,\sigma(n)^{c_n}]
\in \mathbb{Z}_r\wr S_n,
\]
we define
\[
\mathrm{cDes}(g)
:=
\{\, i\in[n] : g(i)\succ g(i+1)\,\},
\]
where we set \(g(n+1):=g(1)\), and the comparison is taken with respect to the
order on colored numbers. Thus, the usual descent set is extended by also
testing the pair \((g(n),g(1))\). 

This notion of cyclic descents goes back to work of Cellini~\cite{cellini-1998-cyclic-eulerian}
(and is often traced, in a broader Lie–theoretic context, to Klyachko), and was further
developed in later works such as~\cite{adin-reiner-roichman-2020-cdes}.
In the present paper cyclic descents play a central role: the invariant that 
we shall construct and use throughout is directly based on them. Since our generating functions are always evaluated at roots of unity of order \(rn\),
the difference between descents and cyclic descents does not affect the results.  
In particular, the flag-major index
remains consistent under this convention.

\subsection{A Coxeter element}

For \(r>2\), the wreath product \(\mathbb{Z}_{r}\wr S_{n}\) is not a
Coxeter group, but it is the well-generated complex reflection group
\(G(r,1,n)\). In this setting, Coxeter elements are precisely the
\(n\)-cycles whose cycle-color is primitive in \(\mathbb{Z}_r\)
\cite[Corollary~5.7]{lewis-wang}.

We shall use the following Coxeter element throughout the paper:
\[
\tau = [\,n^{1},\,1^{0},\,2^{0},\,\dots,\,(n-1)^{0}\,].
\]
Its underlying permutation is an \(n\)-cycle and its cycle-color is \(1\),
so \(\tau\) is indeed a Coxeter element of \(G(r,1,n)\).

The element \(\tau\) will play a central role in our arguments. In
particular, the proof of our main result will be based on the action of
conjugation by \(\tau\) on conjugacy classes and on its effect on the
flag-major index. Observe also that
\[
\tau^n=[1^1,2^1,\ldots,n^1],
\]
and hence \(\tau\) has order \(rn\).

\section{The symmetric group  case}
\begin{theorem}[Garsia, Wachs]
Let $C$ be a conjugacy class of $S_n$,
and let $\omega_n$ be a primitive
$n$-th root of unity. Then
\[
\sum_{\pi \in C} \omega_n^{\mathrm{maj}(\pi)} =
\begin{cases}
\mu(k) & \text{if } C = k^{n/k}, \\
0 & \text{otherwise},
\end{cases}
\]
where \( \mu \) is the Möbius function.
\end{theorem}

This theorem was established by Garsia~\cite{gar90} and 
Wachs~\cite{wachs-1991}.  
In this section we present a new short proof based on conjugation by a Coxeter element.

\begin{proof}
Let \( \tau \in S_n \) be the Coxeter element
\[
\tau = [n,1,2,\dots,n-1].
\]
For any permutation \( \sigma = [\sigma(1), \dots, \sigma(n)] \in S_n \) we have
\[
\sigma \tau = [\sigma(n), \sigma(1), \dots, \sigma(n-1)],
\qquad
\tau^{-1} \sigma = [\sigma(1)+1 \bmod n, \dots, \sigma(n)+1 \bmod n].
\]
Right multiplication by \( \tau \) increases by \(1\) (modulo \(n\)) the position of each cyclic descent.  
Left multiplication by $\tau^{-1}$ preserves all cyclic descents, except for the one involving the value \(n\), which shifts back one position (modulo \(n\)).  
In particular, both operations preserve the number of cyclic descents.

Moreover,
\[
\mathrm{maj}(\sigma \tau) \equiv \mathrm{maj}(\sigma)+\#\mathrm{cDes}(\sigma)\pmod{n},
\qquad
\mathrm{maj}(\tau^{-1}\sigma)\equiv \mathrm{maj}(\sigma)-1\pmod{n}.
\]
Thus for every \(\sigma\in S_n\),
\[
\mathrm{maj}(\tau^{-1}\sigma\tau)\equiv \mathrm{maj}(\sigma)+\#\mathrm{cDes}(\sigma)-1\pmod{n}.
\]

Let \(O\) be an orbit of length \(l > 1\) of this conjugation action, and let \(\sigma\in O\).  
Denote \(d := \#\mathrm{cDes}(\sigma)-1\).  
Then \(\tau^{-l}\sigma\tau^l = \sigma\) and, on the other hand,
\[
    \mathrm{maj}(\tau^{-l}\sigma\tau^l)
    \equiv \mathrm{maj}(\sigma) + dl \pmod{n}.
\]
Hence \(dl\equiv0\pmod{n}\). 
Moreover, \(d \not\equiv 0 \pmod n\). Indeed, otherwise
\(\#\mathrm{cDes}(\sigma)=1\)
and hence
\[
\sigma=[i,\ldots,n,1,\ldots,i-1]
\]
for some \(i\). Therefore, by the computation above, we have
\(\tau^{-1}\sigma\tau=\sigma\).
Thus the orbit of \(\sigma\) has length \(1\), contrary to our assumption. Consequently,
\[
\omega_n^{dl}=1
\qquad\text{and}\qquad
\omega_n^d\neq 1.
\]

Therefore,

\[
\sum_{\pi\in O}\omega_n^{\mathrm{maj}(\pi)}
= \sum_{i=0}^{l-1}\omega_n^{\mathrm{maj}(\tau^{-i}\sigma\tau^{i})}
= \sum_{i=0}^{l-1}\omega_n^{\mathrm{maj}(\sigma)+ i\cdot d} =
\]

\[
\;\omega_n^{\mathrm{maj}(\sigma)}\sum_{i=0}^{l-1}(\omega_n^{d})^{i}
= \omega_n^{\mathrm{maj}(\sigma)}\frac{(\omega_n^{d})^{l}-1}{\omega_n^{d}-1}
= 0.
\]

So every nontrivial orbit contributes zero to the sum.

The trivial orbits are exactly the elements of the cyclic subgroup
\(\langle\tau\rangle\). Indeed, if the orbit of \(\sigma\) is trivial, then
\[
\operatorname{maj}(\tau^{-1}\sigma\tau)\equiv \operatorname{maj}(\sigma)
\pmod n,
\]
and hence \(d\equiv 0\pmod n\). Since \(0\leq d\leq n-2\), we have
\(d=0\), so \(\#\mathrm{cDes}(\sigma)=1\), and therefore
\(\sigma\in\langle\tau\rangle\). Conversely, every power of \(\tau\)
commutes with \(\tau\).
For \(i\in[0,n-1]\), the cycle structure of \(\tau^i\) is \(k^{\,n/k}\), where 
\(k=n/\gcd(i,n)\).  
If \(i=0\), then \(\tau^{0}\) is the identity, which has no descents and hence 
\(\mathrm{maj}(\tau^{0})=0\).  
If \(i>0\), then \(\tau^{i}\) has exactly one cyclic descent, occurring at 
position \(i\), and therefore \(\mathrm{maj}(\tau^{i})=i\).

Let \(C_\lambda\) be a conjugacy class in $S_n$ corresponding to the partition \(\lambda\vdash n\). Since each conjugacy class is invariant under conjugation by \(\tau\), it decomposes into orbits of this action; moreover, every nontrivial orbit contributes zero to the sum. Therefore, it remains only to consider the trivial orbits. That is

\[
\sum_{\pi\in C_\lambda}\omega_n^{\mathrm{maj}(\pi)}
=\sum_{\pi\in C_\lambda\cap\langle\tau\rangle}\omega_n^{\mathrm{maj}(\pi)},
\]

If \(\lambda\neq k^{n/k}\) for any divisor \(k\mid n\), then \(C_\lambda\cap\langle\tau\rangle=\emptyset\) and the sum is zero.  

Otherwise, if \( \lambda = k^{n/k} \), then
\[
C_{\lambda} \cap \langle \tau \rangle
= \left\{ \tau^{i} : i \in [0,n-1],\; \frac{n}{\gcd(i,n)} = k \right\}.
\]

Define
\[
j = \frac{i}{\gcd(i,n)} = i \cdot \frac{k}{n}.
\]
Then \( \gcd(j,k) = 1 \), since:
\[
\gcd(j,k) = \gcd\left( \frac{i}{\gcd(i,n)}, \frac{n}{\gcd(i,n)} \right) = 1.
\]
Thus,
\[
\tau^i = \tau^{j \cdot \frac{n}{k}}, \quad \text{with } \gcd(j,k) = 1,
\]
so
\[
C_{\lambda} \cap \langle \tau \rangle
\subseteq \left\{ \tau^{j \cdot \frac{n}{k}} : j \in [0,k-1],\; \gcd(j,k) = 1 \right\}.
\]

To prove the reverse inclusion, suppose \( j \in [0,k-1] \) with \( \gcd(j,k) = 1 \). Then the order of \( \tau^{j \cdot \frac{n}{k}} \) is:
\[
\frac{n}{\gcd(j \cdot \frac{n}{k},\, n)}
= \frac{n}{\gcd(j \cdot \frac{n}{k},\, \frac{n}{k} \cdot k)}
= \frac{n}{\frac{n}{k} \cdot \gcd(j,k)}
= \frac{n}{\frac{n}{k}} = k,
\]
so the cycle type of \( \tau^{j \cdot \frac{n}{k}} \) is \( k^{n/k} \), i.e., it lies in \( C_\lambda \).

Therefore,
\[
C_{\lambda} \cap \langle \tau \rangle
= \left\{ \tau^{j \cdot \frac{n}{k}} : j \in [0,k-1],\; \gcd(j,k) = 1 \right\}.
\]

Each such element has major index \( \mathrm{maj}(\tau^{j \cdot \frac{n}{k}}) = j \cdot \frac{n}{k} \). Thus,
\begin{align*}
\sum_{\pi \in C_{\lambda}} \omega_n^{\mathrm{maj}(\pi)}
= \sum_{\pi \in C_{\lambda} \cap \langle \tau \rangle} \omega_n^{\mathrm{maj}(\pi)}
= \sum_{\substack{j \in [0,k-1] \\ \gcd(j,k)=1}} \omega_n^{j \cdot \frac{n}{k}}
= \sum_{\substack{j \in [0,k-1] \\ \gcd(j,k)=1}} \omega_k^j
= \mu(k).
\end{align*}
This completes the proof.
\end{proof}

\section{Wreath products}\label{sec:wreath}

In this section we prove the following. Since the case \(r=1\) reduces to the symmetric group \(S_n\), which was treated in Section~3, throughout this section we assume \(r\geq 2\).

\begin{theorem}\label{thm:class-sum-formula}
    Let \(C\) be a conjugacy class of \(\mathbb{Z}_{r}\wr S_{n}\). Then
    \[
        \sum_{g\in C}\omega_{rn}^{\,\mathrm{fmaj}(g)} \;=\;
        \begin{cases}
        \omega_{r}^{c/k} \,\mu(k), & \text{if } C = k_{c}^{\,n/k} \text{ and } \gcd(r,k)=1,\\[4pt]
        0, & \text{otherwise},
        \end{cases}
    \]
   where \(k_c^{\,n/k}\) denotes the conjugacy class in which every cycle has length \(k\) and color \(c\). In the first case, \(c/k\) denotes \(c\,k^{-1}\in\mathbb{Z}_r\), where \(k^{-1}\) is the inverse of \(k\) modulo \(r\).
\end{theorem}

In particular, for the classical Weyl group of type $B_n=\mathbb{Z}_2\wr S_n$ the following holds.

\begin{corollary}
    Let \(C\) be a conjugacy class of $B_n$. Then
    \[
        \sum_{g\in C}\omega_{2n}^{\,\mathrm{fmaj}(g)} \;=\;
        \begin{cases}
        (-1)^{c} \,\mu(k), & \text{if } C = k_{c}^{\,n/k} \text{ and } k \text{ is odd},\\[4pt]
        0, & \text{otherwise},
        \end{cases}
    \]
    where \(k_{c}^{\,n/k}\) denotes the conjugacy class in which every cycle has length \(k\) and color \(c\).  
\end{corollary}

In the symmetric group case, the number of cyclic descents is invariant under
conjugation by the Coxeter element. However, in the wreath product
$\mathbb{Z}_{r} \wr S_{n}$ this invariance no longer holds. The goal of the
next definitions is to isolate a statistic that replaces the cyclic descent
number in this setting, and that remains invariant under conjugation by the
Coxeter element. To this end we introduce the set $S$, the
indicator $\chi_{S}$, and the corrected descent statistic $D$, which will play
a central role in what follows.

\begin{definition}
For a colored number \(a^c\), write
\[
(a^c)^+ := a^{c+1 \bmod r}.
\]
Define the set
\[
S := \{g\in\mathbb{Z}_r\wr S_n :
g(1)\in(g(n)^+,g(n))\},
\]
where \((a,b)\) denotes the open cyclic interval from
\(a\) to \(b\), with respect to the order on colored
numbers defined in Subsection~2.3, traversed in the
increasing direction. In particular, the cyclic interval
\(((a^c)^+,a^c)\) contains exactly \(n-1\) colored numbers.
\end{definition}

\begin{definition}
Define the statistic
\[
\chi_{S} : \mathbb{Z}_{r} \wr S_{n} \longrightarrow \{0,1\},
\]
which takes the value \(1\) for permutations belonging to \(S\), and \(0\) otherwise.
\end{definition}

\begin{definition}
Define the statistic
\[
D : \mathbb{Z}_{r} \wr S_{n} \longrightarrow \mathbb{N}
\quad\text{by}\quad
D(g) = \#\mathrm{cDes}(g) - \chi_{S}(g).
\]
\end{definition}
Since \(n\geq 2\), every permutation of \(n\) distinct entries has at least one and at most \(n-1\) cyclic descents. Hence
\[
0\leq D(g)\leq n-1.
\]

Throughout this section, whenever we refer to the element \(\tau\), 
we mean the Coxeter element defined in Subsection 2.7.  
It will serve as the central object in the statements and proofs that follow, 
and so we do not repeat its definition in each proposition.

We now turn to the effect of conjugation by the element $\tau$ on the
statistics just introduced. In the following propositions we show that
$D$ is invariant under this conjugation, and analyze how the flag major
index transforms under the same operation. These results provide the basic
tools needed for the subsequent arguments.

\begin{proposition}\label{prop:tau-inverse}
Let \( g \in \mathbb{Z}_{r} \wr S_{n} \), and let \(\tau\) denote the Coxeter element.
Then
\[
\mathrm{fmaj}(\tau^{-1} g) \equiv \mathrm{fmaj}(g) - 1 \pmod{rn},
\qquad
D(\tau^{-1} g) = D(g).
\]
\end{proposition}

\begin{proof}
Write
\[
g=[g(1),\ldots,g(n)].
\]
Recall that
\[
\tau^{-1}=[2^0,3^0,\ldots,n^0,1^{r-1}].
\]
By the group multiplication rule, left multiplication by
\(\tau^{-1}\) acts on each colored entry of \(g\) by
applying the map
\[
T(a^c)=
\begin{cases}
(a+1)^c, & a<n,\\
1^{c-1\pmod r}, & a=n.
\end{cases}
\]
Thus,
\[
\tau^{-1}g=[T(g(1)),\ldots,T(g(n))].
\]
Observe that \(T\) acts as a cyclic shift by one on
the ordered set of colored numbers. Hence it preserves
all comparisons, except those involving \(n^0\),
the maximal colored number.

Let \(i\) be the unique position such that
\(g(i)=n^c\). We examine separately the changes
in the major index and in the sum of colors.

\medskip
\noindent
\textbf{Case 1.} Suppose \(c\neq 0\).
Then \(n^0\) does not appear in \(g\), so all
comparisons between consecutive entries are preserved.
Hence
\[
\mathrm{maj}(\tau^{-1}g)=\mathrm{maj}(g),
\qquad
\#\mathrm{cDes}(\tau^{-1}g)=\#\mathrm{cDes}(g).
\]
Moreover, the color of \(g(i)\) decreases by \(1\),
while all other colors remain unchanged. Therefore,
\[
\mathrm{fmaj}(\tau^{-1}g)=\mathrm{fmaj}(g)-1.
\]

\medskip
\noindent
\textbf{Case 2.} Suppose \(c=0\).
Then \(g(i)=n^0\), the maximal colored number,
and \(T(g(i))=1^{r-1}\), the minimal colored number.
Thus, the cyclic descent at position \(i\) disappears,
while a cyclic descent appears at position \(i-1\)
(modulo \(n\)). All other cyclic descents are preserved.
Consequently,
\[
\mathrm{maj}(\tau^{-1}g)
\equiv\mathrm{maj}(g)-1\pmod n,
\qquad
\#\mathrm{cDes}(\tau^{-1}g)=\#\mathrm{cDes}(g).
\]
The color of \(g(i)\) increases from \(0\) to \(r-1\),
while all other colors remain unchanged. Hence
\[
\begin{aligned}
\mathrm{fmaj}(\tau^{-1}g)
&\equiv\mathrm{fmaj}(g)-r+(r-1)\\
&\equiv\mathrm{fmaj}(g)-1\pmod{rn}.
\end{aligned}
\]

\medskip
In both cases, we conclude that
\[
\mathrm{fmaj}(\tau^{-1}g)
\equiv\mathrm{fmaj}(g)-1\pmod{rn},
\qquad
\#\mathrm{cDes}(\tau^{-1}g)=\#\mathrm{cDes}(g).
\]

It remains to show that \(\chi_S\) is invariant.
Since \(T\) preserves cyclic intervals and satisfies
\[
T\bigl((a^c)^+\bigr)=\bigl(T(a^c)\bigr)^+,
\]
we have
\[
\begin{aligned}
g\in S
&\iff g(1)\in(g(n)^+,g(n))\\
&\iff T(g(1))
   \in(T(g(n)^+),T(g(n)))\\
&\iff T(g(1))
   \in(T(g(n))^+,T(g(n)))\\
&\iff \tau^{-1}g\in S.
\end{aligned}
\]
Hence \(\chi_S(\tau^{-1}g)=\chi_S(g)\), and consequently
\[
D(\tau^{-1}g)=D(g).
\]
\end{proof}

\begin{proposition}\label{prop:right-tau}
Let \(g\in\mathbb{Z}_r\wr S_n\), and let
\(\tau\) denote the Coxeter element.
Then
\[
D(g\tau)=D(g).
\]
\end{proposition}

\begin{proof}
Write
\[
g=[g(1),\ldots,g(n)].
\]
Right multiplication by \(\tau\) cyclically rotates
the entries and increases the color of the first
entry by \(1\) modulo \(r\):
\[
g\tau=[g(n)^+,g(1),\ldots,g(n-1)].
\]

Let
\[
I=(g(n)^+,g(n))
\]
be the open cyclic interval with respect to the
order on colored numbers defined in Subsection~2.3.
By definition,
\[
\chi_S(g)=\mathbf{1}_{\{g(1)\in I\}}.
\]

Observe that for any colored number \(u\), the
element \(u^+\) lies exactly \(n\) positions before
\(u\) in the increasing cyclic order.
Consequently, for any two colored numbers \(u,v\),
\[
v\in(u^+,u)
\iff
u^+\in(v^+,v).
\]
Indeed, the first condition means that the forward
cyclic distance from \(v\) to \(u\) is between
\(1\) and \(n-1\). Shifting both elements \(n\)
positions backward preserves this distance, so
\(u^+\) lies among the \(n-1\) elements immediately
following \(v^+\), which is precisely the interval
\((v^+,v)\).

Therefore,
\[
\begin{aligned}
\chi_S(g\tau)=1
&\iff g(n)^+\in(g(n-1)^+,g(n-1))\\
&\iff g(n-1)\in(g(n)^+,g(n))\\
&\iff g(n-1)\in I.
\end{aligned}
\]

Right multiplication by \(\tau\) shifts all cyclic
descent positions forward by one, except possibly
those involving \(g(n)\).
More precisely, the descent conditions at positions
\(n-1\) and \(n\), determined by whether
\[
g(n-1)\succ g(n)
\quad\text{and}\quad
g(n)\succ g(1),
\]
are replaced by the conditions at positions
\(n\) and \(1\), determined by whether
\[
g(n-1)\succ g(n)^+
\quad\text{and}\quad
g(n)^+\succ g(1),
\]
respectively.

We distinguish two cases.

\medskip
\noindent\textbf{Case 1.} \(g(n)^+\prec g(n)\).

The comparison involving \(g(n-1)\) gains a descent
precisely when \(g(n-1)\in I\), while the comparison
involving \(g(1)\) loses a descent precisely when
\(g(1)\in I\). Therefore
\[
\begin{aligned}
&\#\mathrm{cDes}(g\tau)-\#\mathrm{cDes}(g)\\
&\qquad=
\mathbf{1}_{\{g(n-1)\in I\}}
-\mathbf{1}_{\{g(1)\in I\}}\\
&\qquad=\chi_S(g\tau)-\chi_S(g).
\end{aligned}
\]

\medskip
\noindent\textbf{Case 2.} \(g(n)\prec g(n)^+\).

The comparison involving \(g(n-1)\) loses a descent
precisely when \(g(n-1)\notin I\), while the
comparison involving \(g(1)\) gains a descent
precisely when \(g(1)\notin I\). Therefore,
\[
\begin{aligned}
&\#\mathrm{cDes}(g\tau)-\#\mathrm{cDes}(g)\\
&\qquad=
-\mathbf{1}_{\{g(n-1)\notin I\}}
+\mathbf{1}_{\{g(1)\notin I\}}\\
&\qquad=
\mathbf{1}_{\{g(n-1)\in I\}}
-\mathbf{1}_{\{g(1)\in I\}}\\
&\qquad=\chi_S(g\tau)-\chi_S(g).
\end{aligned}
\]

Thus, in both cases,
\[
\#\mathrm{cDes}(g\tau)-\chi_S(g\tau)
=
\#\mathrm{cDes}(g)-\chi_S(g).
\]
By the definition of \(D\),
\[
D(g\tau)=D(g).
\]
\end{proof}

\begin{proposition}\label{prop:fmaj-right-tau}
Let \( g \in \mathbb{Z}_{r} \wr S_{n} \), and let \(\tau\) denote the Coxeter element.  
Then
\[
\mathrm{fmaj}(g \tau) \equiv \mathrm{fmaj}(g) + r \cdot D(g) + 1 \pmod{rn}.
\]
\end{proposition}

\begin{proof}
Write
\[
g = [\sigma(1)^{c_{1}}, \dots, \sigma(n)^{c_{n}}] \in \mathbb{Z}_{r} \wr S_{n}.
\]
Factor \(\tau\) as \(\tau = \tau_{1}\tau_{2}\), where
\[
\tau_{1} = [n^{0}, 1^{0}, \dots, (n-1)^{0}],\qquad
\tau_{2} = [1^{1}, 2^{0}, \dots, n^{0}].
\]
Thus
\[
g \tau = (g \tau_{1}) \tau_{2}.
\]

\paragraph{Step 1: Effect of right multiplication by \(\tau_{1}\).} 
Right multiplication by $\tau_1$ cyclically rotates
the one-line notation without changing the colors:
\[
g\tau_1=
[\sigma(n)^{c_n},\sigma(1)^{c_1},
\ldots,\sigma(n-1)^{c_{n-1}}].
\]

Thus every cyclic descent moves forward by one
position modulo $n$, while the color sum
remains unchanged.

Since a cyclic descent at position $n$
contributes $rn\equiv 0\pmod{rn}$,
we obtain
\[
\operatorname{fmaj}(g\tau_1)
\equiv
\operatorname{fmaj}(g)
+r\cdot\#\operatorname{cDes}(g)
\pmod{rn}.
\]

\paragraph{Step 2: Effect of right multiplication by \(\tau_{2}\).}
Right multiplication by \(\tau_{2}\) increases the color of the entry in the first position by \(1 \bmod r\), leaving all other entries unchanged:
\[
\begin{aligned}
g\tau_1
&=[g(n),g(1),\ldots,g(n-1)],\\
(g\tau_1)\tau_2
&=[\sigma(n)^{c_n+1\bmod r},
\sigma(1)^{c_1},\ldots,\sigma(n-1)^{c_{n-1}}]\\
&=[g(n)^+,g(1),\ldots,g(n-1)].
\end{aligned}
\]
Thus, the descent status at position \(1\) is
determined by comparing \(g(1)\) with \(g(n)\)
and \(g(n)^+\), respectively.

Since a cyclic descent at position \(n\) contributes
\(rn\equiv 0\pmod{rn}\), any change in its status
does not affect \(\mathrm{fmaj}\pmod{rn}\).
Therefore, it suffices to analyze the descent status
at position \(1\), together with the color change
of the first entry.

We distinguish two cases according to whether \(g\in S\). 

\medskip
\noindent\textbf{Case 1.} \(g\in S\).

If \(g(n)^+\prec g(n)\), then
\[
g(n)^+\prec g(1)\prec g(n).
\]
Thus, position \(1\) is a descent of \(g\tau_1\)
but not of \((g\tau_1)\tau_2\), while the color sum
increases by \(1\).

If \(g(n)\prec g(n)^+\), then \(g(1)\) is either
smaller than both or larger than both. Hence the
descent status at position \(1\) is unchanged,
while the color sum decreases by \(r-1\).

In either case,
\[
\mathrm{fmaj}\bigl((g\tau_1)\tau_2\bigr)
\equiv \mathrm{fmaj}(g\tau_1)-r+1
\pmod{rn}.
\]

\medskip
\noindent\textbf{Case 2.} \(g\notin S\).

If \(g(n)^+\prec g(n)\), then \(g(1)\) is either
smaller than both or larger than both. Hence the
descent status at position \(1\) is unchanged,
while the color sum increases by \(1\).

If \(g(n)\prec g(n)^+\), then
\[
g(n)\prec g(1)\prec g(n)^+.
\]
Thus, position \(1\) is a descent of
\((g\tau_1)\tau_2\) but not of \(g\tau_1\),
while the color sum decreases by \(r-1\).

In either case,
\[
\mathrm{fmaj}\bigl((g\tau_1)\tau_2\bigr)
\equiv \mathrm{fmaj}(g\tau_1)+1
\pmod{rn}.
\]

\medskip
Combining both cases, we obtain
\[
\mathrm{fmaj}(g\tau)
\equiv \mathrm{fmaj}(g\tau_1)
-r\chi_S(g)+1
\pmod{rn}.
\]
Therefore, by Step~1,
\[
\begin{aligned}
\mathrm{fmaj}(g\tau)
&\equiv \mathrm{fmaj}(g)
+r\cdot\#\mathrm{cDes}(g)-r\chi_S(g)+1\\
&\equiv \mathrm{fmaj}(g)+r\cdot D(g)+1
\pmod{rn}.
\end{aligned}
\]
This completes the proof.
\end{proof}

Combining Propositions~\ref{prop:tau-inverse},
\ref{prop:right-tau}, and \ref{prop:fmaj-right-tau},
we obtain, for every \(g\in\mathbb{Z}_r\wr S_n\),
\[
D(\tau^{-1}g\tau)=D(g),
\]
and
\[
\mathrm{fmaj}(\tau^{-1}g\tau)
\equiv \mathrm{fmaj}(g)+r\cdot D(g)
\pmod{rn}.
\]
Thus, \(D\) is invariant under conjugation by \(\tau\),
and the change in the flag-major index is determined
by this invariant.

In the next two propositions we characterize
the elements for which \(D(g)=0\) and describe
the distribution of powers of \(\tau\) among the
conjugacy classes of \(\mathbb{Z}_r\wr S_n\).

\begin{proposition}\label{prop:D-zero-iff-power}
Let \( g \in \mathbb{Z}_{r}\wr S_{n} \).  
Then \( D(g) = 0 \) if and only if \( g \) is a power of the Coxeter element \(\tau\).
\end{proposition}

\begin{proof}
(\(\Rightarrow\)) Assume \(D(g)=0\).  
By definition, \(D(g)=\#\mathrm{cDes}(g)-\chi_{S}(g)\), hence 
\(\#\mathrm{cDes}(g)=\chi_{S}(g)\).  
Since \(\chi_{S}\) takes only the values \(0\) or \(1\), and every permutation has at least one cyclic descent, it follows that 
\[
\#\mathrm{cDes}(g)=\chi_{S}(g)=1.
\]
In particular,
\[
g(1)\in(g(n)^+,g(n)).
\]

Let \(p\) be the unique cyclic descent of \(g\). If \(p=n\), then
\[
g(1)\prec g(2)\prec\cdots\prec g(n).
\]
If \(p<n\), then all other cyclic positions are not descents, and therefore
\[
g(p+1)\prec\cdots\prec g(n)\prec
g(1)\prec\cdots\prec g(p).
\]

Thus, in either case, the entries
\(g(1),\ldots,g(n)\) follow the increasing
cyclic order from \(g(1)\) to \(g(n)\), with at
most one wrap around the linear order.
Since
\[
g(1)\in(g(n)^+,g(n)),
\]
it follows that every entry of \(g\) lies in
the half-open cyclic interval
\[
(g(n)^+,g(n)].
\]

As this interval contains exactly \(n\) colored letters, the entries of
\(g\) occupy it entirely. Moreover, since they occur in their cyclic
order, their order in the one-line notation is uniquely determined.

Write \(g(n)=a^c\), with
\(c\in\{0,\ldots,r-1\}\).
Then the one-line notation of \(g\) is
\[
g=\bigl[(a+1)^{\,c+1\!\!\!\pmod r},\,
(a+2)^{\,c+1\!\!\!\pmod r},\dots,\,
n^{\,c+1\!\!\!\pmod r},\,
1^{\,c},\,2^{\,c},\dots,\,a^{\,c}\bigr].
\]
Here the first block is omitted when \(a=n\).

This is precisely the element \(\tau^{\,m}\) with exponent
\[
m = n \cdot c + (n-a).
\]
Indeed, right multiplication by \(\tau\) rotates the positions by one and increases
(mod \(r\)) the color of the leading entry by \(1\); iterating \(n-a\) times
produces the positional block ending at \(a\), and iterating a further
\(nc\) times raises the colors appropriately. Hence \(g=\tau^{m}\), i.e.,
\(g\) is a power of \(\tau\).

\smallskip
(\(\Leftarrow\)) Suppose \(g=\tau^{m}\) for some \(m\).  
By Proposition~\ref{prop:right-tau}, right multiplication by \(\tau\) leaves \(D\) unchanged, 
so all powers of \(\tau\) have the same value of \(D\).  
A direct computation shows that \(D(\tau)=0\), and
Therefore \(D(g)=0\).
\end{proof}

\begin{proposition}\label{prop:powers-of-tau}
Let \( i \in [0,rn-1] \), and let \(\tau\) denote the Coxeter element.  
Set \( d=\gcd(i,n) \), \( k=n/d \), and \( c \equiv i/d \pmod{r} \).  
Then \(\tau^{i}\) lies in the conjugacy class \(k_{c}^{\,n/k}\).  
Moreover,
\[
\mathrm{fmaj}(\tau^{i}) \equiv i \pmod{rn}.
\]
\end{proposition}

\begin{proof}
Since $\tau$ is an $n$-cycle when colors are ignored, the cycle decomposition of $\tau^{i}$ is the same as in $S_{n}$. Therefore, ignoring colors, its cycle structure is
\[
\Bigl(\tfrac{n}{\gcd(i,n)}\Bigr)^{\gcd(i,n)}
\;=\; k^{\,n/k}.
\]

Write \(i=i' + m n\) with \(0 \le i' < n\). Then one checks directly that the one–line notation of \(\tau^{i}\) is
\[
\tau^{i}
=
\bigl[(n-i'+1)^{\,m+1 \bmod r},\,(n-i'+2)^{\,m+1 \bmod r},\dots,\,
n^{\,m+1 \bmod r},\,
1^{\,m},\,2^{\,m},\dots,\,(n-i')^{\,m}\bigr].
\]
Consequently, in $\tau^{i}$ there are exactly $i'$ entries with color \(m+1 \bmod r\), 
and $n-i'$ entries with color $m$.

\medskip

\noindent\textit{Cycle-colors.}
Let $d=\gcd(i,n)$ and $k=n/d$ as above. Each of the $d$ cycles has length $k$.
Indeed, the cycles of the underlying permutation are the residue classes modulo $d$, and since $d\mid i'$, each cycle contains exactly $i'/d$ positions from the first block.
Thus, in each cycle there are $\tfrac{i'}{d}$ entries of color $m+1$ and $\tfrac{n-i'}{d}$ entries of color $m$.
Therefore, the color sum of any cycle is

\[
\frac{i'}{d}\,(m+1) \;+\; \frac{n-i'}{d}\,m
\;\equiv\; \frac{i'+mn}{d}
\;\equiv\; \frac{i}{d}
\;\equiv\; c \pmod{r}.
\]
Hence every cycle has color $c$, that is, $\tau^{i} \in k^{\,n/k}_{c}$.

\medskip
\noindent\textit{Flag-major index.}
We compute $\mathrm{fmaj}(\tau^{i})$ modulo $rn$.

\emph{Case $m\neq r-1$.}
There are no descents, and the color sum equals
\[
(m+1)\,i' + m\,(n-i') \;=\; i' + mn \;=\; i.
\]
Hence $\mathrm{fmaj}(\tau^{i}) \equiv i \pmod{rn}$.

\emph{Case $m = r-1$.} The color sum equals $(r-1)(n-i')$.
If $i'>0$, there is a unique cyclic descent at position $i'$.
If $i'=0$, the unique cyclic descent occurs at position $n$.
In either case, $\operatorname{maj}(\tau^i)=i'$.
Therefore,
\[
\operatorname{fmaj}(\tau^i)
=ri'+(r-1)(n-i')
=i'+(r-1)n=i,
\]
so $\operatorname{fmaj}(\tau^i)\equiv i\pmod{rn}$.

\end{proof}

We first show that the contribution of all nontrivial orbits under the
conjugation action of $\tau$ vanishes. This observation will allow us to
reduce the analysis to the trivial orbits alone.

\begin{proposition}\label{prop:orbit-sum-zero}
Let \(O \subseteq \mathbb{Z}_{r}\wr S_{n}\) be a nontrivial orbit (of length \(l>1\)) 
under conjugation by the Coxeter element \(\tau\). Then
\[
\sum_{g \in O} \omega_{rn}^{\,\mathrm{fmaj}(g)} \;=\; 0.
\]
The same vanishing holds for any choice of primitive \((rn)\)-th root of
unity.
\end{proposition}

\begin{proof}
Fix \(g\in O\). Then
\[
O=\{\tau^{-i} g \tau^{i} : i=0,1,\dots,l-1\}.
\]
By Propositions~\ref{prop:tau-inverse},~\ref{prop:right-tau},~\ref{prop:fmaj-right-tau},
\[
\mathrm{fmaj}(\tau^{-i} g \tau^{i})
\equiv \mathrm{fmaj}(g) + i\cdot r\cdot D(g) \pmod{rn}.
\]
Therefore, letting \(q=\omega_{rn}^{\,r\cdot D(g)}\),
\[
\sum_{x\in O}\omega_{rn}^{\,\mathrm{fmaj}(x)}
=\sum_{i=0}^{l-1}\omega_{rn}^{\,\mathrm{fmaj}(\tau^{-i} g \tau^{i})}
=\omega_{rn}^{\,\mathrm{fmaj}(g)}\sum_{i=0}^{l-1}q^{\,i}.
\]

Since the orbit is nontrivial,
Proposition~\ref{prop:D-zero-iff-power}
implies $D(g)\neq 0$.
Moreover, $0\leq D(g)\leq n-1$,
so $1\leq D(g)\leq n-1$.
Consequently,
$q=\omega_{rn}^{\,r\cdot D(g)}\neq 1$.
Applying the geometric–series formula, we obtain
\[
\sum_{x\in O}\omega_{rn}^{\,\mathrm{fmaj}(x)}
=\omega_{rn}^{\,\mathrm{fmaj}(g)}\cdot\frac{q^{\,l}-1}{q-1}.
\]
It remains to show that \(q^{\,l}=1\).

Since \(\tau^{-l} g \tau^{l}=g\), applying the congruence with \(i=l\) yields
\[
\mathrm{fmaj}(g)\equiv \mathrm{fmaj}(g)+l\cdot r\cdot D(g)\pmod{rn}.
\]
Hence \(l\cdot r\cdot D(g)\equiv 0\pmod{rn}\), so
\[
q^{\,l}=\bigl(\omega_{rn}^{\,r\cdot D(g)}\bigr)^{l}=1.
\]
This completes the proof.
\end{proof}

Having restricted to trivial orbits, we next analyze their structure in
detail. We classify them according to conjugacy classes and compute the
corresponding sums. This leads to explicit evaluations depending on whether
$r$ and $k$ are coprime or not.

\begin{proposition}\label{prop:class-sums-coprime}
Let \(C\) be a conjugacy class of \(\mathbb{Z}_{r}\wr S_{n}\).  
If \(C = k_{c}^{\,n/k}\), i.e. the class in which all cycles have length \(k\) and color \(c\), then
\[
\sum_{g \in C} \omega_{rn}^{\,\mathrm{fmaj}(g)}
= \sum_{j \in J} \omega_{rk}^{\,j},
\]
where
\[
J=\{\,j \in [0,rk-1] : \gcd(j,k)=1,\;\; j \equiv c \pmod{r}\,\}.
\]
Otherwise,
\[
\sum_{g \in C} \omega_{rn}^{\,\mathrm{fmaj}(g)} = 0.
\]
\end{proposition}

\begin{proof}
The class \(C\) decomposes into orbits under conjugation by \(\tau\).
By Proposition~\ref{prop:orbit-sum-zero}, the sum of \(\omega_{rn}^{\,\mathrm{fmaj}(g)}\) over any nontrivial orbit vanishes.  
By Propositions~\ref{prop:tau-inverse} and ~\ref{prop:fmaj-right-tau}, every trivial orbit
satisfies $rD(g)\equiv 0\pmod{rn}$.
Since $0\leq D(g)\leq n-1$, this implies $D(g)=0$.
By Proposition~\ref{prop:D-zero-iff-power},
$g\in\langle\tau\rangle$.
Conversely, every power of $\tau$ commutes with $\tau$.
Hence the trivial orbits correspond precisely
to the elements of $\langle\tau\rangle$.  
Hence
\[
\sum_{g \in C}\omega_{rn}^{\,\mathrm{fmaj}(g)}
= \sum_{g \in C \cap \langle \tau\rangle} \omega_{rn}^{\,\mathrm{fmaj}(g)}.
\]

If \(C\) is not of the form \(k_{c}^{\,n/k}\), then Proposition~\ref{prop:powers-of-tau} implies that
\(C \cap \langle \tau\rangle = \varnothing\), and the sum is zero.  

Now assume \(C=k_{c}^{\,n/k}\).  
By Proposition~\ref{prop:powers-of-tau}, \(\tau^{i}\in C\) if and only if
\[
k = \frac{n}{\gcd(i,n)} \quad\text{and}\quad c \equiv \frac{i}{\gcd(i,n)} \pmod{r}.
\]
Setting
\[
j = \frac{i}{\gcd(i,n)} = i \cdot \frac{k}{n},
\]
we have \(j \equiv c \pmod{r}\) and \(\gcd(j,k)=1\).  
Thus
\[
C \cap \langle \tau \rangle \subseteq \{\,\tau^{\,j\cdot n/k} : j \in J\,\}.
\]

For the reverse inclusion, let \(j \in J\).  
Then the cycle structure of \(\tau^{\,j\cdot n/k}\) has length
\[
\frac{n}{\gcd(j\cdot n/k,n)}
= \frac{n}{\gcd(j\cdot n/k,k\cdot n/k)}
= \frac{n}{\tfrac{n}{k}\cdot \gcd(j,k)}
= k,
\]
and the cycle color is
\[
\frac{j\cdot n/k}{\gcd(j\cdot n/k,n)}
= \frac{j\cdot n/k}{\tfrac{n}{k}\cdot \gcd(j,k)}
= j \equiv c \pmod{r}.
\]
Hence $\tau^{\,j\cdot n/k} \in C$.  

\smallskip
By Proposition~\ref{prop:powers-of-tau} we have
\[
\mathrm{fmaj}\!\left(\tau^{\,j\cdot n/k}\right) \equiv j\cdot \tfrac{n}{k} \pmod{rn}.
\]
Moreover,
\[
\omega_{rn}^{\,j\cdot n/k}
= \left(\omega_{rk\cdot (n/k)}^{\,n/k}\right)^{j}
= \omega_{rk}^{\,j}.
\]

Therefore
\[
\sum_{g \in C}\omega_{rn}^{\,\mathrm{fmaj}(g)}
= \sum_{g \in C \cap \langle\tau\rangle}\omega_{rn}^{\,\mathrm{fmaj}(g)}
= \sum_{\tau^{\,j\cdot n/k} :\, j \in J}\omega_{rn}^{\,\mathrm{fmaj}(\tau^{\,j\cdot n/k})}
= \sum_{j \in J}\omega_{rn}^{\,j\cdot n/k}
= \sum_{j \in J}\omega_{rk}^{\,j}.
\]

\end{proof}

\begin{proposition}\label{prop:almostlast}
Let \(C = k_{c}^{\,n/k}\) be a conjugacy class of \(\mathbb{Z}_{r}\wr S_{n}\).  
If \(\gcd(r,k)=1\), then
\[
\sum_{g \in C}\omega_{rn}^{\,\mathrm{fmaj}(g)}
= \omega_{r}^{\,c'} \,\mu(k),
\]

where \(c' = c\,k^{-1}\) in \(\mathbb{Z}_r\), with
\(k^{-1}\) denoting the inverse of \(k\) modulo \(r\).

\end{proposition}

\begin{proof}
By Proposition ~\ref{prop:class-sums-coprime}, it suffices to evaluate
\[
\sum_{j \in J}\omega_{rk}^{\,j},
\qquad
J=\{\,j \in [0,rk-1] : \gcd(j,k)=1,\; j \equiv c \pmod{r}\,\}.
\]

Since \(\gcd(r,k)=1\), there exists \(c' \in [0,r-1]\) such that
\[
c \equiv c'k \pmod{r}.
\]
Thus every \(j \in J\) can be represented modulo \(rk\) in the form
\[
j = r m + c'k, \qquad m \in [0,k-1],
\]
and the condition \(\gcd(j,k)=1\) is equivalent to \(\gcd(m,k)=1\).  
Hence, modulo \(rk\)
\[
J = \{\, r m + c'k : m \in [0,k-1],\; \gcd(m,k)=1 \,\}.
\]

Therefore,
\[
\sum_{g \in C}\omega_{rn}^{\,\mathrm{fmaj}(g)}
= \sum_{j \in J}\omega_{rk}^{\,j}
= \sum_{\substack{m \in [0,k-1]\\ \gcd(m,k)=1}}\omega_{rk}^{\,r m + c'k}
= \omega_{r}^{\,c'} \sum_{\substack{m \in [0,k-1]\\ \gcd(m,k)=1}} \omega_{k}^{\,m}.
\]

Since the terms are precisely the primitive \(k\)-th roots of unity, their sum equals \(\mu(k)\).
Hence
\[
\sum_{g \in C}\omega_{rn}^{\,\mathrm{fmaj}(g)}
= \omega_{r}^{\,c'} \,\mu(k).
\]

\end{proof}

\begin{proposition}\label{prop:last}
Let \(C = k_{c}^{\,n/k}\) be a conjugacy class of \(\mathbb{Z}_{r}\wr S_{n}\).  
If \(\gcd(r,k) > 1\), then
\[
\sum_{g \in C}\omega_{rn}^{\,\mathrm{fmaj}(g)} = 0.
\]
\end{proposition}

\begin{proof}
By Proposition~\ref{prop:class-sums-coprime}, it suffices to evaluate
\[
\sum_{j \in J}\omega_{rk}^{\,j},
\qquad
J = \{\, j \in [0,rk-1] : \gcd(j,k)=1,\; j \equiv c \pmod{r}\,\}.
\]

Let \(l=\mathrm{lcm}(r,k)\).  
For any \(j\in J\), we have \(\gcd(j+l,k)=1\) (since \(l\) is a multiple of \(k\) and \(j\) is coprime to \(k\)), and also \(j+l\equiv j \pmod{r}\).  
Thus \(j+l \pmod{rk}\in J\), and because \(l<rk\), the element \(j+l\) is distinct from \(j\).  
It follows that \(J\) decomposes into nontrivial orbits under the shift map \(j\mapsto j+l\pmod{rk}\).

The length of each such orbit is \(d=\gcd(r,k)\).  
Let \(O\subseteq J\) be one such orbit, say
\[
O = \{\,j_{0}+ ml \pmod{rk}\ : m=0,1,\dots,d-1 \,\}.
\]
Then
\[
\sum_{j\in O}\omega_{rk}^{\,j}
= \sum_{m=0}^{d-1}\omega_{rk}^{\,j_{0}+ m l}
= \omega_{rk}^{\,j_{0}}\sum_{m=0}^{d-1} (\omega_{rk}^{\,l})^{m}.
\]

Since \(l=\mathrm{lcm}(r,k)=\tfrac{rk}{d}\), we have
\[
\omega_{rk}^{\,l} = \omega_{rk}^{\,rk/d} = \omega_{d}.
\]
Therefore,
\[
\sum_{j\in O}\omega_{rk}^{\,j}
= \omega_{rk}^{\,j_{0}}\sum_{m=0}^{d-1}\omega_{d}^{\,m}.
\]
But \(\sum_{m=0}^{d-1}\omega_{d}^{\,m}=0\).  
Hence every orbit \(O\) contributes zero, and thus
\[
\sum_{g \in C}\omega_{rn}^{\,\mathrm{fmaj}(g)}=0.
\]
\end{proof}

Bringing together all previous propositions, we can now prove 
the main result of this section. 

\begin{proof}[Proof of Theorem~\ref{thm:class-sum-formula}]
There are three cases to consider.  
If \(C\) is not of the special form \(k_{c}^{\,n/k}\), then the sum is \(0\), as established in Proposition~\ref{prop:class-sums-coprime}.  
If \(C=k_{c}^{\,n/k}\) with \(\gcd(r,k)>1\), the sum is again \(0\), by Proposition~\ref{prop:last}.  
Finally, if \(C=k_{c}^{\,n/k}\) with \(\gcd(r,k)=1\), the evaluation is given by Proposition~\ref{prop:almostlast}.  
\end{proof}

\section*{Acknowledgments}

This paper is part of the author's ongoing Ph.D. research at Bar-Ilan University, under the supervision of Ron M. Adin and Yuval Roichman. The author is deeply grateful to both for their guidance, insightful discussions, and continued support. The author would also like to thank his fellow graduate students for stimulating mathematical discussions.

\printbibliography
\end{document}